\documentclass[VANCOUVER,STIX1COL]{WileyNJD-v2}
\usepackage{moreverb}

\newtheorem{thm}{Theorem}

\newtheorem{rmk}{Remark}

\newtheorem{alg}{Algorithm}

\newcommand\BibTeX{{\rmfamily B\kern-.05em \textsc{i\kern-.025em b}\kern-.08em
T\kern-.1667em\lower.7ex\hbox{E}\kern-.125emX}}

\articletype{Article Type}%

\received{<day> <Month>, <year>}
\revised{<day> <Month>, <year>}
\accepted{<day> <Month>, <year>}

\begin{document}

\title{A novel greedy Gauss-Seidel method for solving large linear least squares problem\protect}

\author{Yanjun Zhang}

\author{Hanyu Li*}

\authormark{ZHANG AND LI}

\address{\orgdiv{College of Mathematics and Statistics}, \orgname{Chongqing University}, \orgaddress{\state{Chongqing}, \country{China}}}

\corres{*Hanyu Li, College of Mathematics and Statistics, Chongqing University, Chongqing 401331, P.R. China.\\ \email{lihy.hy@gmail.com or hyli@cqu.edu.cn.}}

\presentaddress{National Natural Science Foundation of China, Grant/Award Number: 11671060; Natural Science Foundation Project of CQ CSTC, Grant/Award Number: cstc2019jcyj-msxmX0267}

\abstract[Summary]{We present a novel greedy Gauss-Seidel method for solving large linear least squares problem. This method improves the greedy randomized coordinate descent (GRCD) method proposed recently by Bai and Wu 
[Bai ZZ, and Wu WT. On greedy randomized coordinate descent methods for solving large linear least-squares problems. Numer Linear Algebra Appl. 2019;26(4):1--15], which in turn improves the popular randomized Gauss-Seidel method. Convergence analysis of the new method is provided. Numerical experiments show that, for the same accuracy, our method outperforms the GRCD method in term of the computing time.}

\keywords{greedy Gauss-Seidel method, greedy randomized coordinate descent method, randomized Gauss-Seidel method, large linear least squares problem}


\maketitle


\section{Introduction}\label{sec1}
Linear least squares problem is a classical linear algebra problem in scientific computing, arising for instance in many parameter estimation problems. In the literature, several direct methods for solving this problem are studied. Such methods including the use of QR factorization with pivoting and the use of singular value decomposition (SVD) \cite{Osborne1996,Higham2002} require high storage and 
are expensive when the matrix 
 is large-scale. Hence, iterative methods are considered for solving large linear least squares problem, such as the famous Gauss-Seidel method \cite{Saad2003}.

Inspired by a work of Strohmer and Vershynin \cite{Strohmer2009} which shows that the randomized Kaczmarz method converges linearly in expectation to the solution, Leventhal and Lewis \cite{Leventhal2010} obtained a similar result for the randomized Gauss-Seidel (RGS) method, which is also called the randomized coordinate descent method. This method works on the columns of the matrix $A$ to minimize $ \|\mathbf{b}-\mathbf{A} \mathbf{x}\|^2_{2}$ randomly according to an appropriate probability distribution and has attracted much attention recently due to its better performance; see for example \cite{Ma2015,Edalatpour2017,Hefny2017,Tu2017,Chen2017,Tian2017,Xu1,Dukui2019,Razaviyayn2019} and references therein.



Recently, Bai and Wu \cite{Bai2019} proposed a greedy randomized coordinate descent (GRCD) method by introducing an efficient probability criterion for selecting the working columns from the matrix $A$, which avoids a weakness of the one adopted in the RGS method. The GRCD method is faster than the RGS method in terms of the number of iterations and computing time. By the way, the idea of greed applied in \cite{Bai2019}  has wide applications, see for example \cite{Griebel2012,Nguyen2017,Bai2018,Bai2018r,Nutini2018,Zhang2019,Du2019,Liu2019} and references therein.


In the present paper, we develop a novel greedy Gauss-Seidel (GGS) method for solving large linear least squares problem, which adopts a quite different way to determine the working columns of the matrix $A$ compared with the GRCD method and hence needs less computing time in each iteration; see the detailed analysis before Algorithm \ref{alg2} below. In theory, we prove the convergence of the GGS method. In numerical experiments, we compare the performance of the GGS and GRCD methods using the examples from \cite{Bai2019}. Numerical results show that, for the same accuracy, the GGS method requires almost the same number of iterations as that of the GRCD method, however,  the GGS method spends less computing time 
 in all the cases.

The rest of this paper is organized as follows. In Section \ref{sec2}, notation and some preliminaries are provided. We present our novel GGS method and its convergence properties in Section \ref{sec3}. Numerical experiments are given in Section \ref{sec4}. 

\section{Notation and Preliminaries }\label{sec2}

For a vector $z\in R^{n}$, $z^{(j)}$ represents its $j$th entry. For a matrix $G=(g_{ij})\in R^{m\times n}$, $G_{(j)}$, $\|G\|_2$, and $\|G\|_F$ denote its $j$th column, spectral norm, and Frobenius norm, respectively. Moreover, if the matrix $G\in R^{n\times n}$ is positive definite, then we define the energy norm of any vector $x\in R^{n}$ as $\| x\|_G:=\sqrt{x^TGx}$, where $(\cdot)^T$ denotes the transpose of a vector or a matrix. In addition, we denote the identity matrix by $I$, its $j$th column by $e_j$, the smallest positive eigenvalue of $G^{T}G$ by $\lambda_{\min}\left( G^{T} G\right)$ and the number of elements of a set $\mathcal{W}$ by $|\mathcal{W}|$.

In what follows, as done in \cite{Bai2019}, we use $x_{\star}=A^{\dag}b$, with $A^{\dag}=(A^TA)^{-1}A^T$ being the Moore-Penrose pseudoinverse, to denote the unique least squares solution to the linear least squares problem:
\begin{equation}
\label{ls}
\min \limits _{\mathbf{x} \in \mathbb{R}^{n}}\|\mathbf{b}-\mathbf{A} \mathbf{x}\|^2_{2},
\end{equation}
where $A\in R^{m\times n}$ is of full column rank and $b\in R^{m}$.
As we know, the solution $x_\star:=\texttt{arg}\min \limits _{\mathbf{x} \in \mathbb{R}^{n}}\|\mathbf{b}-\mathbf{A} \mathbf{x}\|^2_{2}$ is the solution to the following normal equation \cite{Osborne1961} for (\ref{ls}):
\begin{equation}
\label{1}
A^TAx=A^Tb.
\end{equation}
Based on the normal equation (\ref{1}), Bai and Wu \cite{Bai2019} proposed the GRCD method listed as follows, where $r_k = b-Ax_k$ denotes the residual vector.
\begin{alg}
\label{alg1}
 The GRCD method
\begin{enumerate}[]
\item \mbox{INPUT:} ~$A\in R^{m\times n}$, $b\in R^{m}$, $\ell$ , initial estimate $x_0$
\item \mbox{OUTPUT:} ~$x_\ell$
\item For $k=0, 1, 2, \ldots, \ell-1$ do
\item ~~~~Compute
 $$\delta_{k}=\frac{1}{2}\left(\frac{1}{\left\|A^Tr_k\right\|_{2}^{2}} \max \limits _{1 \leq j  \leq n}\left\{\frac{\left|A^T_{(j)}r_k\right|^{2}}{\left\|A_{\left(j\right)}\right\|_{2}^{2}}\right\}+\frac{1 }{\|A\|_{F}^{2}}\right).$$
\item ~~~~Determine the index set of positive integers
 $$\mathcal{V}_{k}=\left\{j \Bigg|  \left|A^T_{(j)}r_k\right|^{2} \geq \delta_{k} \left\|A^Tr_k\right\|_{2}^{2} \left\|A_{\left(j\right)}\right\|_{2}^{2}\right\}.$$
\item ~~~~Let $s_k=A^Tr_k$ and define $\tilde{s}_k$ as follows
  $$
\tilde{s}_{k}^{(j)}=\left\{\begin{array}{ll}{s_{k}^{(j)},} & {\text { if } j \in \mathcal{V}_{k}}, \\ {0,} & {\text { otherwise. }}\end{array}\right.
$$
\item ~~~~Select $j_k \in \mathcal{V}_{k}$ with probability Pr(column = $j_k$)=$\frac{|\tilde{s}_{k}^{(j_k)}|^2}{\|\tilde{s}_{k}\|^2_2}$.
\item ~~~~Set
$$x_{k+1}=x_{k}+\frac{ s_{k}^{(j_k)} }{ \| A_{\left(j_{k}\right)} \|_{2}^{2}}e_{j_k}.$$
\item End for
\end{enumerate}
\end{alg}

From the definitions of  $\delta_{k}$ and $\mathcal{V}_{k}$ in Algorithm \ref{alg1}, we have  that if $\ell\in \mathcal{V}_{k}$, then
\begin{eqnarray*}
\frac{\left|A^T_{(\ell)}r_k\right|^{2}}{\left\|A_{\left(\ell\right)}\right\|_{2}^{2}}\geq\frac{1}{2}\left( \max \limits _{1 \leq j  \leq n}\left\{\frac{\left|A^T_{(j)}r_k\right|^{2}}{\left\|A_{\left(j\right)}\right\|_{2}^{2}}\right\}+\frac{\left\|A^Tr_k\right\|_{2}^{2} }{\|A\|_{F}^{2}}\right).
\end{eqnarray*}
Note that
\begin{equation*}
\label{3}
\max \limits _{1 \leq j  \leq n}\left\{\frac{\left|A^T_{(j)}r_k\right|^{2}}{\left\|A_{\left(j\right)}\right\|_{2}^{2}}\right\}\geq\sum_{j=1}^{n}\frac{\|A_{(j)}\|^2_2}{\| A\|^2_F}\frac{ \left|A^T_{(j)}r_k\right|^{2}}{\|A_{(j)}\|^2_2} \notag= \frac{\left\|A^Tr_k\right\|^{2}_2}{\|A\|_{F}^{2}}.
\end{equation*}
Thus, we can't conclude that if $\ell\in \mathcal{V}_{k}$, then
\begin{eqnarray*}
\frac{\left|A^T_{(\ell)}r_k\right|^{2}}{\left\|A_{\left(\ell\right)}\right\|_{2}^{2}}\geq \max \limits _{1 \leq j  \leq n}\left\{\frac{\left|A^T_{(j)}r_k\right|^{2}}{\left\|A_{\left(j\right)}\right\|_{2}^{2}}\right\}, \textrm{ i.e., } \ \frac{\left|A^T_{(\ell)}r_k\right|^{2}}{\left\|A_{\left(\ell\right)}\right\|_{2}^{2}}= \max \limits _{1 \leq j  \leq n}\left\{\frac{\left|A^T_{(j)}r_k\right|^{2}}{\left\|A_{\left(j\right)}\right\|_{2}^{2}}\right\}.
\end{eqnarray*}
As a result, there may exist some $\ell\in \mathcal{V}_{k}$ such that
\begin{equation}
\label{3343434}
 \frac{\left|A^T_{(\ell)}r_k\right|^{2}}{\left\|A_{\left(\ell\right)}\right\|_{2}^{2}}<\max \limits _{1 \leq j  \leq n}\left\{\frac{\left|A^T_{(j)}r_k\right|^{2}}{\left\|A_{\left(j\right)}\right\|_{2}^{2}}\right\}.
\end{equation}
Meanwhile, from the update formula, for any $j_k\in \mathcal{V}_{k}$, we have
\begin{equation}
\label{2}
\|Ax_{k+1}-Ax_{k}\|^2_2=\frac{\left|A^T_{(j_k)}r_k\right|^{2}}{\left\|A_{\left(j_k\right)}\right\|_{2}^{2}}.
\end{equation}
Thus, combining (\ref{3343434}) and (\ref{2}), we can find that we can't make sure any column with the index from the index set $\mathcal{V}_{k}$ make the distance between $Ax_{k+1}$ and $Ax_{k}$ be the  largest when finding $x_{k+1}$. Furthermore, to compute $\delta_{k}$, we have to calculate the norm of each column of the matrix $A$.

\section{A Novel Greedy Gauss-Seidel Method}\label{sec3}

Considering that a column with the index from the index set $\mathcal{V}_{k}$ in the GRCD method may make the distance between $Ax_{k+1}$ and $Ax_{k}$ not be the largest and to compute $\delta_{k}$ needs to calculate the norm of each column of the matrix $A$, and
inspired by
 some recent works on selection strategy for working index based on the maximum residual \cite{Nutini2018,Haddock2019,Rebrova2019}, we design a new method which  includes two main steps. In the first step, we use the maximum entries of the residual vector $s_k$ of the normal equation (\ref{1}) to determine an index set $\mathcal{R}_{k}$ whose specific definition is given in Algorithm \ref{alg2}. In the second step, we capture an index from the set $\mathcal{R}_{k}$ with which we can make sure the distance between $Ax_{k+1}$ and $Ax_{k}$ be the largest for any possible $x_{k+1}$. On a high level, the new method seems to change the order of the two main steps of Algorithm \ref{alg1}. However, comparing with the GRCD method, besides making the distance between $Ax_{k+1}$ and $Ax_{k}$ always be the largest when finding $x_{k+1}$, 
 we also do not need to calculate the norm of each column of the matrix $A$ any longer in Algorithm \ref{alg2}. Moreover, we can also find that the number of elements in set $\mathcal{R}_{k}$ may be less than the number of elements in set $\mathcal{V}_{k}$, i.e., $|\mathcal{R}_{k}|<|\mathcal{V}_{k}|$ because $\mathcal{R}_{k}$ is determined by the maximum entries of the vector $s_k$. Consequently,  our method can reduce the computation cost at each iteration and hence behaves better in the computing time, which is confirmed by extensive numerical experiments given in Section \ref{sec4}.

 Based on the above introduction, we propose the following algorithm, i.e., Algorithm \ref{alg2}.

\begin{alg}
\label{alg2}
 The GGS method
\begin{enumerate}[]
\item \mbox{INPUT:} ~$A\in R^{m\times n}$, $b\in R^{m}$, $\ell$ , initial estimate $x_0$
\item \mbox{OUTPUT:} ~$x_\ell$
\item For $k=0, 1, 2, \ldots, \ell-1$ do
\item ~~~~Determine the index set of positive integers
 $$\mathcal{R}_{k}=\left\{\tilde{j}_{k}\Bigg| \tilde{j}_{k}= {\rm arg} \max \limits _{1 \leq j  \leq n}\left|A^T_{(j)}r_k\right|\right\}.$$
\item ~~~~Compute
$$j_{k}={\rm arg} \max \limits _{\tilde{j}_{k}\in \mathcal{R}_{k}}\left\{\frac{\left|A^T_{(\tilde{j}_{k})}r_k\right|^2 }{\left\|A_{(\tilde{j}_{k})}\right\|^2_{2}  }\right\}.$$
\item ~~~~Set
$$x_{k+1}=x_{k}+\frac{ A^T_{(j_{k})}r_k}{ \| A_{\left(j_{k}\right)} \|_{2}^{2}}e_{j_k}.$$
\item End for
\end{enumerate}
\end{alg}

\begin{rmk}
\label{rmk}
Note that if
$$\left|A^T_{(j_k)}r_k\right|=  \max \limits _{1 \leq j  \leq n}\left|A^T_{(j)}r_k\right|,$$
then $j_k\in\mathcal{R}_{k}.$ So the index set $\mathcal{R}_{k}$ in Algorithm \ref{alg2} is nonempty for all iteration index $k$.
\end{rmk}

\begin{rmk}
\label{rmk110}
Like Algorithm \ref{alg1}, we can use the values of $\frac{\left|A^T_{(\tilde{j}_{k})}r_k\right|^2 }{\left\|A_{(\tilde{j}_{k})}\right\|^2_{2}  }$ for $\tilde{j}_{k}\in \mathcal{R}_{k}$ as a  probability selection criterion to devise a randomized version of Algorithm \ref{alg2}. In this case, the convergence factor may be a little worse than that of Algorithm \ref{alg2} because, for the latter,  the index is selected based on the largest value of $\frac{\left|A^T_{(\tilde{j}_{k})}r_k\right|^2 }{\left\|A_{(\tilde{j}_{k})}\right\|^2_{2}  }$ for $\tilde{j}_{k}\in \mathcal{R}_{k}$, which make the distance between $Ax_{k+1}$ and $Ax_{k}$ be the largest for any possible $x_{k+1}$.
\end{rmk}



In the following, we give the convergence theorem of the GGS method.

\begin{thm}
\label{theorem1}
The iteration sequence $\{x_k\}_{k=0}^\infty$ generated by Algorithm \ref{alg2}, starting from an initial guess $x_0\in R^{n}$, converges linearly to the unique least squares solution $x_{\star}=A^{\dag}b$ and
\begin{equation}
\label{4}
  \| x_1-x_\star\|^2_{A^TA} \leq\left(1-\frac{1}{|\mathcal{R}_{0}|} \cdot\frac{1}{\sum\limits_{j_0\in \mathcal{R}_{0} }\|A_{(j_0)}\|^2_2 }\cdot\frac{1}{n}\cdot\lambda_{\min}\left( A^{T} A\right)\right)\| x_0-x_\star\|^2_{A^TA},
\end{equation}
and
 \begin{equation}
\label{5}
\| x_{k+1}-x_\star\|^2_{A^TA} \leq\left(1-\frac{1}{|\mathcal{R}_{k}|} \cdot\frac{1}{\sum\limits_{j_k\in \mathcal{R}_{k}}\|A_{(j_k)}\|^2_2 }\cdot\frac{1}{n-1}\cdot\lambda_{\min}\left( A^{T} A\right) \right)\| x_k-x_\star\|^2_{A^TA},~~k=1, 2, \ldots .
\end{equation}
Moreover, let $\alpha=\max \{|\mathcal{R}_{k}|\}$, $\beta=\max\{\sum\limits_{j_k\in \mathcal{R}_{k}}\|A_{(j_k)}\|^2_2\}, k=0, 1, 2, \ldots.$ Then,
\begin{equation}
\label{6}
\| x_{k}-x_\star\|^2_{A^TA} \leq\left(1-\frac{\lambda_{\min}\left( A^{T} A\right) }{\alpha\cdot \beta\cdot(n-1)} \right)^{k-1}\left(1- \frac{\lambda_{\min}\left( A^{T} A\right)}{|\mathcal{R}_{0}| \cdot\sum\limits_{j_0\in \mathcal{R}_{0}}\|A_{(j_0)}\|^2_2 \cdot n  }\right)\cdot\| x_0-x_\star\|^2_{A^TA},~~k=1, 2, \ldots .
\end{equation}
\end{thm}

\begin{proof}
From the update rule in Algorithm \ref{alg2}, we have
 $$A(x_{k+1}-x_{k})=\frac{ A^T_{(j_{k})}r_k }{ \| A_{\left(j_{k}\right)} \|_{2}^{2}}A_{\left(j_{k}\right)},$$
which implies that $A(x_{k+1}-x_{k})$ is parallel to $A_{\left(j_{k}\right)}$. Meanwhile,
\begin{eqnarray*}
A(x_{k+1}-x_{\star})&=&A\left(x_{k}-x_\star+\frac{ A^T_{(j_{k})}r_k }{ \| A_{\left(j_{k}\right)} \|_{2}^{2}}e_{j_k}\right)
\\
&=& A\left(x_{k}-x_\star\right)+\frac{ A^T_{(j_{k})}r_k }{ \| A_{\left(j_{k}\right)} \|_{2}^{2}}A_{\left(j_{k}\right)},
\end{eqnarray*}
which together with the fact $A^TAx_{\star}=A^Tb$ gives
\begin{eqnarray*}
A(x_{k+1}-x_{\star})&=& \left(I-\frac{A_{(j_{k})} A^T_{(j_{k})}  }{ \| A_{\left(j_{k}\right)} \|_{2}^{2}} \right)A\left(x_{k}-x_\star\right).
\end{eqnarray*}
Then
\begin{eqnarray*}
A^T_{(j_{k})} A(x_{k+1}-x_{\star})=A^T_{(j_{k})}  \left(I-\frac{A_{(j_{k})} A^T_{(j_{k})}  }{ \| A_{\left(j_{k}\right)} \|_{2}^{2}} \right)A\left(x_{k}-x_\star\right)=0,
\end{eqnarray*}
and hence $A(x_{k+1}-x_{\star})$ is orthogonal to $A_{\left(j_{k}\right)}$. Thus, the vector $A(x_{k+1}-x_{k})$ is perpendicular to the vector $A(x_{k+1}-x_{\star})$. By the Pythagorean theorem, we get
\begin{equation*}
\|A(x_{k+1}-x_{\star})\|^2_2=\|A(x_{k}-x_{\star})\|^2_2-\|A(x_{k+1}-x_{k})\|^2_2,
\end{equation*}
or equivalently,
\begin{equation}
\label{7}
\left\|x_{k+1}-x_{\star}\right\|^2_{A^{T}A}=\left\|x_{k}-x_{\star}\right\|^2_{A^{T}A}-\left\|x_{k+1}-x_{k}\right\|^2_{A^{T}A}.
\end{equation}
On the other hand, from Algorithm \ref{alg2},  we have
$$\left|A^T_{(j_{k})}r_k\right|=\max\limits_{1\leq j\leq n}  \left|A^T_{(j)}r_k\right|
~{\rm and}~
\frac{ \left|A^T_{(j_{k})}r_k\right|^2 }{ \left\| A_{\left(j_{k}\right)}\right \|_{2}^{2}}=\max\limits_{j\in \mathcal{R}_{k}}\frac{ \left|A^T_{(j)}r_k\right|^2 }{ \left\| A_{\left(j\right)}\right \|_{2}^{2}}.
$$
Then
\begin{align}
\left\|x_{k+1}-x_{k}\right\|^2_{A^{T}A}&=~\|A(x_{k+1}-x_{k})\|^2_2=\frac{ \left|A^T_{(j_{k})}r_k\right|^2 }{ \left\| A_{\left(j_{k}\right)}\right \|_{2}^{2}} \geq~ \sum\limits_{j_k\in \mathcal{R}_{k}} \frac{\frac{ \left|A^T_{(j_{k})}r_k\right|^2 }{ \left\| A_{\left(j_{k}\right)}\right \|_{2}^{2}}}{\sum\limits_{j\in \mathcal{R}_{k}}\frac{ \left|A^T_{(j)}r_k\right|^2 }{ \left\| A_{\left(j\right)}\right \|_{2}^{2}} }\cdot \frac{ \left|A^T_{(j_{k})}r_k\right|^2 }{ \left\| A_{\left(j_{k}\right)}\right \|_{2}^{2}}  \notag\\
&\geq ~\sum\limits_{j_k\in \mathcal{R}_{k}} \frac{1}{|\mathcal{R}_{k}|}\cdot\frac{ \left|A^T_{(j_{k})}r_k\right|^2 }{ \left\| A_{\left(j_{k}\right)}\right \|_{2}^{2}} = ~\sum\limits_{j_k\in \mathcal{R}_{k}} \frac{1}{|\mathcal{R}_{k}|}\cdot\frac{ \max\limits_{1\leq j\leq n}  \left|A^T_{(j)}r_k\right|^2}{ \left\| A_{\left(j_{k}\right)}\right \|_{2}^{2}}.\label{8}
\end{align}
Thus, substituting (\ref{8}) into (\ref{7}), we obtain
\begin{equation}
\label{71212}
\left\|x_{k+1}-x_{\star}\right\|^2_{A^{T}A}\leq\left\|x_{k}-x_{\star}\right\|^2_{A^{T}A}-\sum\limits_{j_k\in \mathcal{R}_{k}} \frac{1}{|\mathcal{R}_{k}|}\cdot\frac{ \max\limits_{1\leq j\leq n}  \left|A^T_{(j)}r_k\right|^2}{ \left\| A_{\left(j_{k}\right)}\right \|_{2}^{2}} .
\end{equation}

For $k=0$, we have
\begin{align*}
\max\limits_{1\leq j\leq n}\left|A^T_{(j)}r_0\right|^2&= ~\max\limits_{1\leq j\leq n}\left|A^T_{(j)}r_0\right|^2 \cdot \frac{\left\|A^T r_0\right\|^2_2}{\sum\limits_{j=1}^{n}\left|A^T_{(j)}r_0\right|^2}\geq~\frac{1}{n}\cdot\left\|A^T r_0\right\|^2_2,
\end{align*}
which together with a result from \cite{Bai2018}:
\begin{align}
\|A^{T}x\|^2_2\geq\lambda_{\min}\left( A^{T} A\right)\|x\|^2_2  \label{1210}
\end{align}
is valid for any vector $x$ in the column space of $A$, implies
\begin{align}
\max\limits_{1\leq j\leq n}\left|A^T_{(j)}r_0\right|^2&\geq~\frac{1}{n}\cdot\lambda_{\min}\left( A^{T} A\right) \cdot\left\| Ax_\star-Ax_0\right\|^2_2\notag\\
&=~\frac{1}{n}\cdot\lambda_{\min}\left( A^{T} A\right) \cdot\left\|x_0-x_\star\right\|^2_{A^{T}A}. \label{10}
\end{align}
Thus, substituting (\ref{10}) into (\ref{71212}), we obtain
\begin{align}
\left\|x_{1}-x_{\star}\right\|^2_{A^{T}A}&\leq~\left\|x_{0}-x_{\star}\right\|^2_{A^{T}A}-\sum\limits_{j_0\in\mathcal{R}_{0}} \frac{1}{|\mathcal{R}_{0}|}\cdot\frac{ 1}{ \left\| A_{\left(j_{0}\right)}\right \|_{2}^{2}} \cdot \frac{1}{n}\cdot\lambda_{\min}\left( A^{T} A\right) \cdot\left\|x_0-x_\star\right\|^2_{A^{T}A} \notag\\
&=~\left(1-\frac{1}{|\mathcal{R}_{0}|}\cdot\frac{1}{\sum\limits_{j_0\in \mathcal{R}_{0}}\|A_{(j_0)}\|^2_2 }\cdot\frac{1}{n} \cdot\lambda_{\min}\left( A^{T} A\right)\right)\cdot\| x_0-x_\star\|^2_{A^{T}A}, \notag
\end{align}
which is just the estimate (\ref{4}).

%
For $k\geq1$, we have
\begin{align*}
\max\limits_{1\leq j\leq n}\left|A^T_{(j)}r_k\right|^2&= ~\max\limits_{1\leq j\leq n}\left|A^T_{(j)}r_k\right|^2 \cdot \frac{\left\|A^T r_k\right\|^2_2}{\sum\limits_{j=1}^{n}\left|A^T_{(j)}r_k\right|^2}.
\end{align*}
Note that, according to the update formula in Algorithm \ref{alg2}, it is easy to obtain
\begin{align}
 A^T_{(j_{k-1})}r_k&=  ~A^T_{(j_{k-1})}\left(r_{k-1}- \frac{ A^T_{(j_{k-1})}r_{k-1} }{ \| A_{\left(j_{k-1}\right)} \|_{2}^{2}}A_{\left(j_{k-1}\right)} \right)   \notag\\
&= ~A^T_{(j_{k-1})}\left(r_{k-1}\right)- A^T_{(j_{k-1})}\left(r_{k-1}\right)  =~0. \label{9}
\end{align}
Then
\begin{align*}
\max\limits_{1\leq j\leq n}\left|A^T_{(j)}r_k\right|^2&=~\max\limits_{1\leq j\leq n}\left|A^T_{(j)}r_k\right|^2 \cdot \frac{\left\|A^T r_k\right\|^2_2}{\sum\limits _{j=1 \atop j \neq j_{k-1}}^{n}\left|A^T_{(j)}r_k\right|^2}\geq~ \frac{1}{n-1}\cdot\left\|A^Tr_k\right\|^2_2,
\end{align*}
which together with  (\ref{1210}) yields
\begin{align}
\max\limits_{1\leq j\leq n}\left|A^T_{(j)}r_k\right|^2&\geq~ \frac{1}{n-1}\cdot\lambda_{\min}\left( A^{T} A\right)\left\|Ax_\star-Ax_k\right\|^2_2\notag\\
&=~ \frac{1}{n-1}\cdot\lambda_{\min}\left( A^{T} A\right)\left\|x_k-x_\star\right\|^2_{A^TA}.\label{12}
\end{align}
Thus, substituting (\ref{12}) into (\ref{71212}), we get
\begin{align}
\left\|x_{k+1}-x_{\star}\right\|^2_{A^{T}A}&\leq~\left\|x_{k}-x_{\star}\right\|^2_{A^{T}A}-\sum\limits_{j_k\in \mathcal{R}_{k}} \frac{1}{|\mathcal{R}_{k}|}\cdot  \frac{ 1}{ \left\| A_{\left(j_{k}\right)}\right \|_{2}^{2}}\cdot \frac{1}{n-1}\cdot\lambda_{\min}\left( A^{T} A\right)\left\|x_k-x_\star\right\|^2_{A^TA}  \notag\\
&=~\left(1-\frac{1}{|\mathcal{R}_{k}|}\cdot\frac{1}{\sum\limits_{j_k\in \mathcal{R}_{k}}\|A_{(j_k)}\|^2_2 }\cdot\frac{1}{n-1} \cdot\lambda_{\min}\left( A^{T} A\right) \right)\| x_k-x_\star\|^2_{A^TA}.
\end{align}
So the estimate (\ref{5}) is obtained.
%
By induction on the iteration index $k$, we have the estimate (\ref{6}).
\end{proof}
\begin{rmk}
\label{rmk1}
Since $1\leq \alpha \leq n$ and $\min\limits_{1\leq j\leq n}\|A_{(j)}\|^{2}_2\leq \beta \leq \|A\|^{2}_{F}$, it holds that
$$\left(1-\frac{\lambda_{\min}\left( A^{T} A\right) }{\min\limits_{1\leq j\leq n}\|A_{(j)}\|^{2}_2 \cdot (n-1) } \right) \leq\left(1-\frac{\lambda_{\min}\left( A^{T} A\right) }{\alpha\cdot \beta\cdot(n-1)} \right)\leq\left(1-\frac{\lambda_{\min}\left( A^{T} A\right) }{n\cdot \|A\|^{2}_{F}\cdot(n-1)} \right).$$
Hence, the convergence factor of the GGS method is small when the parameters $\alpha$ and $\beta$ are small. So, the smaller size of $|\mathcal{R}_{k}|$ is,  the better convergence factor of the GGS method is when $\beta$ is fixed. From the analysis before Algorithm \ref{alg2}, we know that the size of $|\mathcal{R}_{k}|$ may be smaller than that of $|\mathcal{V}_{k}|$. This is one of the reasons that our algorithm behaves better in the computing time.
\end{rmk}
\begin{rmk}\label{rmk12111}


 If $\alpha=1$ and $ \beta=\min\limits_{1\leq j\leq n}\|A_{(j)}\|^{2}_2$,  the right side of (\ref{5}) is smaller than
 $$\left(1-\frac{1}{\min\limits_{1\leq j\leq n}\|A_{(j)}\|^{2}_2 \cdot (n-1)}\lambda_{\min}\left( A^{T} A\right)\right)\left\|x_{k}-x_{\star}\right\|_{A^{T} A}^{2}.$$
Since
 $$ \min\limits_{1\leq j\leq n}\|A_{(j)}\|^{2}_2 \cdot (n-1)\leq\|A\|^2_F-\min\limits_{1\leq j\leq n}\|A_{(j)}\|^{2}_2<\|A\|^2_F,$$
which implies
 $$\frac{1}{\min\limits_{1\leq j\leq n}\|A_{(j)}\|^{2}_2 \cdot (n-1)}>\frac{1}{2} \left(\frac{1}{\|A\|_{F}^{2}-\min \limits_{1 \leq j \leq n}\left\|A_{(j)}\right\|_{2}^{2}}+\frac{1}{\|A\|_{F}^{2}}\right ),$$
we have
\begin{align*}
&\left(1-\frac{1}{\min\limits_{1\leq j\leq n}\|A_{(j)}\|^{2}_2 \cdot (n-1)}\lambda_{\min}\left( A^{T} A\right)\right)\left\|x_{k}-x_{\star}\right\|_{A^{T} A}^{2}\\
&<\left(1-\frac{1}{2}\left (\frac{1}{\|A\|_{F}^{2}-\min \limits_{1 \leq j \leq n}\left\|A_{(j)}\right\|_{2}^{2}}+\frac{1}{\|A\|_{F}^{2}} \right)\lambda_{\min}\left( A^{T} A\right)\right)\left\|x_{k}-x_{\star}\right\|_{A^{T} A}^{2}.
\end{align*}
Note that the error estimate in expectation of the GRCD method in \cite{Bai2019} is
$$
\mathbb{E}_{k}\left\|x_{k+1}-x_{\star}\right\|_{A^{T} A}^{2} \leq\left (1- \frac{1}{2} \left(\frac{1}{\|A\|_{F}^{2}-\min \limits_{1 \leq j \leq n}\left\|A_{(j)}\right\|_{2}^{2}}+\frac{1}{\|A\|_{F}^{2}}   \right)\lambda_{\min } (A^{T} A)\right  )\left\|x_{k}-x_{\star}\right\|_{A^{T} A}^{2} , $$
where $ k=1,2, \ldots.$
So the convergence factor of GGS method is slightly better for the above case.
\end{rmk}

\section{Numerical Experiments}\label{sec4}
In this section, we report the numerical results of the GGS and GRCD methods for solving the linear least squares problem with the matrix $A\in R^{m\times n}$ from two sets. One is generated randomly by using the MATLAB function \texttt{randn}, and the other includes some sparse matrices originating in different applications from \cite{Davis2011}. 
To compare the GGS and GRCD methods fairly and directly, we use the examples from \cite{Bai2019}.

We compare the two methods mainly in terms of the iteration numbers (denoted as ``IT'') and the computing time in seconds (denoted as ``CPU''), and the IT and CPU listed in our numerical results denote the arithmetical averages of the required iteration numbers and the elapsed CPU times with respect to 50 times repeated runs of the corresponding methods. 
Furthermore, to give an intuitive compare of the two methods, we also present the iteration number speed-up of the GGS method against the GRCD method, which is defined as
\begin{eqnarray*}
\texttt{IT speed-up}=\frac{\texttt{IT of GRCD } }{\texttt{IT of GGS } },
\end{eqnarray*}
and the computing time speed-up of the GGS method against the GRCD method, which is defined as
\begin{eqnarray*}
\texttt{CPU speed-up}=\frac{\texttt{CPU of GRCD} }{\texttt{CPU of GGS} }.
\end{eqnarray*}
In addition, for the sparse matrices from \cite{Davis2011}, we define the density as follows
\begin{eqnarray*}
\texttt{density}=\frac{\texttt{number of nonzero of an $m\times n$ matrix}}{\texttt{mn}},
\end{eqnarray*}
and use \texttt{cond(A)} to represent the Euclidean condition number of the matrix $A$. 

In our specific experiments, the solution vector $x_\star$ is generated randomly by the MATLAB function \texttt{randn}. For the consistent problem, we set the right-hand side $b=Ax_{\star}$. For the inconsistent problem, we set the right-hand side $b=Ax_{\star}+r_{0}$, where $r_0$ is a nonzero vector belonging to the null space of $A^{T}$, which is generated by the MATLAB function \texttt{null}. All  the test problems are started from an initial zero vector $x_{0}=0$ and terminated once the \emph{relative solution error} (\texttt{RES}), defined by $$\texttt{RES}=\frac{\left\|x_{k}-x_{\star}\right\|^{2}_2}{\left\|x_{\star}\right\|^{2}_2},$$ satisfies $\texttt{RES}\leq10^{-6}$ or the number of iteration steps exceeds $ 200,000$.

\begin{table}[tp]
  \centering
  \fontsize{6.5}{8}\selectfont
    \caption{ Numerical results for the \texttt{GGS} and \texttt{GRCD} methods when the system is consistent.}
    \label{table1}
    \begin{tabular}{|c|c|c|c|c|c|c|}
    \hline
    \multirow{2}{*}{$m\times n$}&
    \multicolumn{3}{c|}{IT}&\multicolumn{3}{c|}{CPU}\cr\cline{2-7}
    &GGS&GRCD&IT speed-up& GGS&GRCD&CPU speed-up\cr
\hline
$1000\times 50$&   126.0000 & 128.2400  &  1.0178  &  0.0138  &  0.0631  &  4.5909 \cr
$1000\times 100$&  374.0000 & 361.5000  &  0.9666  &  0.0466  &  0.1703  &  3.6577\cr
$1000\times 150$&  603.0000 & 600.5600  &  0.9960  &  0.1044  &  0.3194  &  3.0599\cr\hline

$2000\times 50$&   108.0000 & 106.2600  &  0.9839  &  0.0125  &  0.0525  &  4.2000 \cr
$2000\times 100$&  246.0000 &245.7200   & 0.9989   & 0.0466   & 0.1313   & 2.8188\cr
$2000\times 150$&  439.0000 & 445.6800  &  1.0152  &  0.1094  &  0.2691  &  2.4600\cr\hline

$3000\times 50$&   105.0000 & 104.9600  &  0.9996  &  0.0172  &  0.0556  &  3.2364\cr
$3000\times 100$&  231.0000 & 236.8800  &  1.0255  &  0.0619  &  0.1444  &  2.3333\cr
$3000\times 150$&  409.0000 & 409.0400  &  1.0001  &  0.1400  &  0.2834  &  2.0246\cr\hline

$4000\times 50$&   96.0000  & 99.7400   & 1.0390   & 0.0194   & 0.0572   & 2.9516 \cr
$4000\times 100$&  205.0000 & 209.1200  &  1.0201  &  0.0678  &  0.1388  &  2.0461\cr
$4000\times 150$&  337.0000 & 343.6600  &  1.0198  &  0.1638  &  0.2662  &  \textbf{1.6260}\cr\hline

$5000\times 50$&   96.0000  & 95.3800   & 0.9935   & 0.0250   & 0.0600   & 2.4000 \cr
$5000\times 100$&  195.0000 & 203.0800  &  1.0414  &  0.0728  &  0.1569  &  2.1545\cr
$5000\times 150$&  340.0000 & 337.0200  &  0.9912  &  0.1819  &  0.2978  &  1.6375\cr\hline
   \end{tabular}
\end{table}
\begin{table}[tp]
  \centering
  \fontsize{6.5}{8}\selectfont
    \caption{ Numerical results for the \texttt{GGS} and \texttt{GRCD} methods when the system is inconsistent.}
    \label{table2}
    \begin{tabular}{|c|c|c|c|c|c|c|}
    \hline
    \multirow{2}{*}{$m\times n$}&
    \multicolumn{3}{c|}{IT}&\multicolumn{3}{c|}{CPU}\cr\cline{2-7}
    &GGS&GRCD&IT speed-up& GGS&GRCD&CPU speed-up\cr
\hline
$1000\times 50$&   120.0000 & 124.8600  &  1.0405  &  0.0125  &  0.0591  &  4.7250 \cr
$1000\times 100$&  329.0000 & 321.3800  &  0.9768  &  0.0400  &  0.1591  &  3.9766\cr
$1000\times 150$&  589.0000 & 579.5600  &  0.9840  &  0.0994  &  0.3009  &  3.0283\cr\hline

$2000\times 50$&   113.0000 & 110.2000  &  0.9752  &  0.0119  &  0.0566  &  \textbf{4.7632} \cr
$2000\times 100$&  245.0000 & 250.0600  &  1.0207  &  0.0531  &  0.1322  &  2.4882\cr
$2000\times 150$&  434.0000 & 444.7200  &  1.0247  &  0.1113  &  0.2666  &  2.3961\cr\hline

$3000\times 50$&   107.0000 & 105.0800  &  0.9821  &  0.0194  &  0.0553  &  2.8548\cr
$3000\times 100$&  235.0000 & 232.3600  &  0.9888  &  0.0609  & 0.1412   & 2.3179\cr
$3000\times 150$&  399.0000 & 401.4600  &  1.0062  &  0.1403  &  0.2769  &  1.9733\cr\hline

$4000\times 50$&   95.0000  & 97.4800   & 1.0261   & 0.0194   & 0.0537   & 2.7742 \cr
$4000\times 100$&  220.0000 & 216.7400  &  0.9852  &  0.0694  &  0.1444  &  2.0811\cr
$4000\times 150$&  348.0000 & 356.8000  &  1.0253  &  0.1525  &  0.2772  &  1.8176\cr\hline

$5000\times 50$&   87.0000  & 91.9400   & 1.0568   & 0.0187   & 0.0559   & 2.9833 \cr
$5000\times 100$&  212.0000 & 215.9600  &  1.0187  &  0.0862  &  0.1566  &  1.8152\cr
$5000\times 150$&  336.0000 & 339.2600  &  1.0097  &  0.1641  &  0.3050  &  1.8590\cr\hline
   \end{tabular}
\end{table}
For the first class of matrices, that is, the randomly generated matrices, the numerical results on IT and CPU are listed in Table \ref{table1} when the linear system is consistent, and in Table \ref{table2} when the linear system is inconsistent. From Tables \ref{table1} and \ref{table2}, we see that the GGS method requires almost the same number of iterations as that of the GRCD method, but the GGS method is more efficient in term of the computing time. The computing time speed-up is at least 1.626 (see Table \ref{table1} for the $4000\times 150$ matrix) and at most 4.7632 (see Table \ref{table2} for the $2000\times  50$ matrix).
\begin{table}[!htbp]\centering
\begin{small}\scriptsize
\caption{Numerical results for the \texttt{GGS} and \texttt{GRCD} methods when the system is consistent.}\centering \label{Table3}
 \begin{tabular}{ccccccccccccccccccc}
 \hline
\textbf{name}  &&\textbf{abtahal}     & \textbf{Cities} &\textbf{divorce}&\textbf{ WorldCities}&\textbf{ Trefethen\_300} &\textbf{cage5} &  \\
\hline
 $m \times n$  && $14596 \times 209 $ & $55 \times 46 $ & $50 \times 9$  &$315 \times 100$    &$300 \times 300 $        &$37 \times 37 $ &  \\

density        && 1.68\%             &    53.04\%      & 50.00\%        & 23.87\%            &5.20\%                   &17.02\%       &    \\

\texttt{cond(A)} &&  12.23            &   207.15          & 19.39         & 66.00              &    1772.69              &15.42           &   \\
 \hline
IT            &\texttt{GGS}&  14888            &  29181         & 634         & 5011              & 3210            &1477         &   \\
              &\texttt{GRCD}&  13966           &  40937         &  647       & 5011            &   1374            & 1624.4           &   \\
 & \texttt{IT speed-up } &  0.9380            &   1.4029         & 1.0200        & 1.0000      & 0.4280       & 1.0998         &   \\
\hline
CPU            &\texttt{GGS}&   8.2550          &  0.1747       & 0.0028         &0.0772             &  0.0416            &0.0066          &   \\
               &\texttt{GRCD}&   12.6428        &  1.8497       & 0.0316        & 0.2916              & 0.0734             &0.0700          &   \\
  &\texttt{CPU speed-up }&  \textbf{1.5315}        &   10.5886           & \textbf{11.2222}        & 3.7773            &   1.7669             &10.6667           &   \\
\hline
\end{tabular}
\end{small}
\end{table}

\begin{table}[!htbp]\centering
\begin{small}\scriptsize
\caption{Numerical results for the \texttt{GGS} and \texttt{GRCD} methods when the system is inconsistent.}\centering \label{Table4}
 \begin{tabular}{ccccccccccccccccc}
 \hline
\textbf{name}  &&\textbf{abtahal}     & \textbf{Cities} &\textbf{divorce}&\textbf{ WorldCities}& \\
\hline
 $m \times n$  && $14596 \times 209 $ & $55 \times 46 $ & $50 \times 9$  &$315 \times 100$    &  \\

density        && 1.68\%             &    53.04\%      & 50.00\%        & 23.87\%            &  \\

\texttt{cond(A)} &&  12.23            &   207.15          & 19.39         & 66.00              &     \\
 \hline
IT            &\texttt{GGS}&  11264            &  28449         & 552        &  3532            &  \\
              &\texttt{GRCD}& 12571          &   39752         &   496.6800      &  3576.2           &     \\
 &\texttt{IT speed-up }  &   1.1160           &   1.3973        & 0.8998        &  1.0125      &   \\
\hline
CPU            &\texttt{GGS}&   6.2750         &  0.1716      & 0.0028         &0.0550             &  \\
               &\texttt{GRCD}&  11.3034       &  1.8278      &  0.0213       &  0.2050             & \\
&\texttt{CPU speed-up }  &  1.8013       &  10.6539         & 7.5556      &   3.7273         &   \\
\hline
\end{tabular}
\end{small}
\end{table}

For the second class of matrices, that is, the sparse full column rank matrices from \cite{Davis2011}, the numerical results on IT and CPU are listed in Table \ref{Table3} when the linear system is consistent, and in Table \ref{Table4} when the linear system is inconsistent. In both tables, the iteration numbers of the GGS and GRCD methods are almost the same except for the case of the matrix \textbf{Trefethen\_300}, which is very ill-conditioned.  But for all the matrices, the CPUs of the GGS method are smaller than those of the GRCD method, with the CPU speed-up being at least 1.5315 (the matrix abtahal in Table \ref{Table3}) and at most 11.2222 (the matrix divorce in Table \ref{Table3}).

Therefore, in all the cases, although the GGS method requires almost the same number of iterations as that of the GRCD method except for a very special case, the former 
outperforms the latter in term of the computing time, which is consistent with the analysis before Algorithm \ref{alg2}.


\bibliography{mybibfile}

\end{document}